\documentclass[reqno,10pt]{amsart}
\usepackage{amssymb, amsthm, amsmath}
\usepackage{version}
\usepackage{fancybox}
\usepackage{bm}
\usepackage{color}
\usepackage{mathrsfs}
\newtheorem{definition}{\sc Definition}[section]
\newtheorem{theorem}[definition]{\sc Theorem}
\newtheorem{lemma}[definition]{\sc Lemma}

\newtheorem{proposition}[definition]{\sc Proposition}

\newtheorem{remark}[definition]{\sc Remark}

\makeatletter

\@addtoreset{equation}{section}

\renewcommand{\d}{\text{\rm d}}
\newcommand{\vep}{\varepsilon}

\newcommand{\vm}{|v|^{m-2}v}
\newcommand{\vmh}{|v|^{(m-2)/2}v}
\newcommand{\vnm}{|v_n|^{m-2}v_n}
\newcommand{\vnmh}{|v_n|^{(m-2)/2}v_n}
\newcommand{\sigr}{|\sigma|^{r-2}\sigma}
\newcommand{\sigm}{|\sigma|^{m-2}\sigma}
\newcommand{\sigmh}{|\sigma|^{(m-2)/2}\sigma}
\newcommand{\sigmcomp}{|\sigma|^{\frac{m(r-1)}2 - 1}\sigma}

\newcommand{\III}{\color{black}}
\newcommand{\EEE}{\color{black}}

\begin{document}


\title[Stability of non-isolated asymptotic profiles for fast
diffusion]{Stability of non-isolated asymptotic profiles\\for fast diffusion}

\author{Goro Akagi}
\address[Goro Akagi]{Graduate School of System Informatics, Kobe University,
1-1 Rokkodai-cho, Nada-ku, Kobe 657-8501 Japan}
\email{akagi@port.kobe-u.ac.jp}

\date{\today}

\keywords{Fast diffusion equation, stability analysis, non-isolated
stationary points, \L ojasiewicz-Simon inequality}

\begin{abstract}
The stability of asymptotic profiles of solutions to the
 Cauchy-Dirichlet problem for Fast Diffusion Equation (FDE, for short)
 is discussed. The main result of the present paper is the stability of
 any asymptotic profiles of least energy. It is noteworthy that this
 result can cover non-isolated profiles, e.g., those for thin annular
 domain cases. The method of proof is based on the \L ojasiewicz-Simon
 inequality, which is usually used to prove the convergence of solutions
 to prescribed limits, as well as a uniform extinction estimate for
 solutions to FDE. Besides, local minimizers of an energy functional
 associated with this issue are characterized. Furthermore, the
 instability of positive radial asymptotic profiles in thin annular
 domains is also proved by applying the \L ojasiewicz-Simon inequality in a different way.
\end{abstract}

\subjclass[2010]{\emph{Primary}: 35K67; \emph{Secondary}: 35B40, 35B35} 
\maketitle

\section{Introduction}

Let $\Omega$ be a bounded domain of $\mathbb R^N$ with smooth boundary
$\partial \Omega$. We are concerned with the Cauchy-Dirichlet problem
for Fast Diffusion Equation (shortly, FDE) of the form
\begin{alignat}{4}
\partial_t \left( |u|^{m-2}u \right) &= \Delta u
 \quad && \mbox{ in } \Omega \times (0, \infty),\label{eq:1.1}\\
 u &= 0 && \mbox{ on } \partial \Omega \times (0, \infty), \label{eq:1.2}\\
 u|_{t=0} &= u_0 && \mbox{ in } \Omega,\label{eq:1.3}
\end{alignat}
where $\partial_t = \partial/\partial t$, under the assumptions that
\begin{equation}\label{hypo}
u_0 \in H^1_0(\Omega), \quad 2 < m < 2^* := \dfrac{2N}{(N-2)_+}.
\end{equation}
FDE arises in plasma physics to describe anomalous diffusion of plasma
in a Tokamak, a toroidal device to confine plasma by imposing a magnetic
field (see~\cite{BH78,BH,BH82} and~\cite{Okuda-Dawson}). One of typical
features of solutions to \eqref{eq:1.1}--\eqref{eq:1.3} is the
extinction in finite time, namely, every solution vanishes at a finite time
(see~\cite{Sabinina62,BC,Diaz88,HerreroVazquez88}). Moreover, Berryman
and Holland~\cite{BH} determined the optimal extinction rate
of  solutions  $u = u(x,t)$ vanishing at a finite
time $t_* = t_*(u_0)$  under \eqref{hypo}. More precisely, it holds that
$$
c_1 (t_* - t)_+^{1/(m-2)} \leq \|u(\cdot,t)\|_{H^1_0(\Omega)} 
\leq c_2 (t_* - t)_+^{1/(m-2)}
\quad  \mbox{ for all } \ t \geq 0
$$
with $c_1, c_2 > 0$, provided that $u_0 \not\equiv 0$.
Here and henceforth, we write $\|u\|_{H^1_0(\Omega)} = \|\nabla
u\|_{L^2(\Omega)} = (\int_\Omega |\nabla u(x)|^2 \,\d x)^{1/2}$. 
Furthermore, they also proved the existence of \emph{asymptotic profiles} of
vanishing solutions, that is, a nonzero limit of the rescaled solution $(t_* -
t)^{-1/(m-2)} u(x,t)$ along a sequence $t_n \nearrow t_*$ (see
also~\cite{Kwong88-3,DKV91,SavareVespri}
and~\cite{BBDGV09,BDGV10,BGV11,BGV08,BV10}).

In order to characterize the asymptotic profile of $u(x,t)$, apply the
change of variable, 
$$
v(x,s) = (t_* - t)^{-1/(m-2)} u(x,t)
\quad \mbox{ with } \  s = \log (t_*/(t_* - t)).
$$
Then $v(x,s)$ solves
\begin{alignat}{4}
 \partial_s \left( |v|^{m-2}v \right) - \Delta v &= \lambda_m |v|^{m-2}v
\quad && \mbox{ in } \Omega \times (0, \infty),\label{eq:1.6}\\
 v &= 0 && \mbox{ on } \partial \Omega \times (0, \infty), \label{eq:1.7}\\
 v|_{s = 0} &= v_0 && \mbox{ in } \Omega\label{eq:1.8}
\end{alignat}
with $\lambda_m := (m-1)/(m-2) > 0$ and the initial data $v_0 :=
t_*(u_0)^{-1/(m-2)}u_0$.
Each asymptotic profile can be regarded as the limit of $v(x,s)$ along
a subsequence $s_n \nearrow \infty$; therefore, profiles are
characterized as nontrivial solutions to the stationary problem,
\begin{alignat}{4}
 - \Delta \phi &= \lambda_m |\phi|^{m-2}\phi \quad && \mbox{ in } \Omega,
\label{eq:1.10}\\
\phi &= 0 && \mbox{ on } \partial \Omega.
\label{eq:1.11}
\end{alignat}
On the other hand, each nontrivial solution $\phi(x)$ of \eqref{eq:1.10},
\eqref{eq:1.11} forms a separable solution $U(x,t) := (1 -
t)_+^{1/(m-2)} \phi(x)$ to \eqref{eq:1.1}--\eqref{eq:1.3}, and then,
$U(x,0) = \phi(x)$, $t_*(\phi) = 1$ and $\phi(x)$ is the asymptotic profile of
$U(x,t)$. Therefore the set of
all nontrivial solutions to \eqref{eq:1.10}, \eqref{eq:1.11} coincides
with the set of all asymptotic profiles for
\eqref{eq:1.1}--\eqref{eq:1.3}. From now on, we denote this set by
$\mathcal S$.

This paper addresses the \emph{stability of asymptotic profiles} for
FDE, that is, whether or not solutions of
\eqref{eq:1.1}--\eqref{eq:1.3} emanating from a small neighborhood (in
$H^1_0(\Omega)$) of
an asymptotic profile $\phi \in \mathcal S$ also have the same profile
$\phi$. Such a notion of stability has been formulated in~\cite{AK09} by
introducing a dynamical system generated by
\eqref{eq:1.6}--\eqref{eq:1.8} in a peculiar phase set
$$
\mathcal X := \{t_*(u_0)^{-1/(m-2)}u_0 \colon u_0 \in H^1_0(\Omega)
\setminus \{0\}\},
$$
which is equivalently rewritten by $\mathcal X = \{v_0 \in H^1_0(\Omega)
\colon t_*(v_0) = 1\}$ (hence $\mathcal S \subset \mathcal X$) and
homeomorphic to the unit sphere in $H^1_0(\Omega)$
(see~\cite[Propositions 6 and 10]{AK09}). More precisely, it is
defined as follows:
\begin{definition}[Stability and instability of asymptotic profiles~\cite{AK09}]\label{D:stbl}
Let $\phi \in \mathcal S$.
\begin{enumerate}
\item $\phi$ is said to be \emph{stable}, if
      for any $\varepsilon>0$ there exists $\delta > 0$ such
      that any solution $v$ of \eqref{eq:1.6}, \eqref{eq:1.7} satisfies
	       \begin{equation*}
		\sup_{s \in [0, \infty)}
		 \|v(\cdot,s) -\phi\|_{H^1_0(\Omega)} < \varepsilon,
	       \end{equation*}
	     whenever $v(\cdot,0) \in \mathcal X$ and $\|v(\cdot,0) -
      \phi\|_{H^1_0(\Omega)}<\delta$.
\item $\phi$ is said to be \emph{unstable}, if $\phi$ is not stable.
\item $\phi$ is said to be \emph{asymptotically stable}, if $\phi$ is stable,
      and moreover, there exists $\delta_0 > 0$
      such that any solution $v$ of \eqref{eq:1.6}, \eqref{eq:1.7}
      satisfies
      \begin{equation*}
       \lim_{s \nearrow \infty}\|v(\cdot,s) - \phi\|_{H^1_0(\Omega)} = 0,
      \end{equation*}
      whenever $v(\cdot,0) \in \mathcal X$ and $\|v(\cdot,0) - \phi\|_{H^1_0(\Omega)}<\delta_0$.
\end{enumerate}
\end{definition}
In~\cite{AK09}, some stability criteria are also established for
\emph{isolated} profiles (see Proposition \ref{P:criteria} below). Here
\emph{least energy solutions} to \eqref{eq:1.10}, \eqref{eq:1.11}
mean nontrivial solutions achieving the \emph{least energy}, that is,
the infimum over $\mathcal S$ of the \emph{energy} functional $J :
H^1_0(\Omega) \to \mathbb R$ defined by
$$
J(w) := \dfrac 1 2 \int_\Omega |\nabla w(x)|^2 \,\d x
 - \dfrac{\lambda_m} m \int_\Omega |w(x)|^m \,\d x
 \quad \mbox{ for } \ w \in H^1_0(\Omega)
$$
(see, e.g.,~\cite{Rabinowitz} for more details).
Least energy solutions of \eqref{eq:1.10}, \eqref{eq:1.11} turn out to be
sign-definite by strong maximum principle. We also note that $J$ is an action
functional associated with \eqref{eq:1.10}, \eqref{eq:1.11} and also a
Lyapunov functional for \eqref{eq:1.6}--\eqref{eq:1.8}.
\begin{proposition}[Stability criteria for isolated asymptotic profiles~\cite{AK09}]\label{P:criteria}
The following {\rm (i)} and {\rm (ii)} hold true\/{\rm :}
\begin{enumerate}
 \item Let $\phi$ be a least energy solution to \eqref{eq:1.10},
       \eqref{eq:1.11} which is isolated {\rm (}in $H^1_0(\Omega)${\rm )} from
       all the other least energy solutions. Then $\phi$ is stable {\rm
       (}in the sense of Definition \ref{D:stbl}{\rm )}. In addition, if
       $\phi$ is isolated from all the other sign-definite solutions,
       $\phi$ is asymptotically stable.
 \item Sign-changing solutions $\psi$ to \eqref{eq:1.10}, \eqref{eq:1.11} are
       not asymptotically stable. In addition, if $\psi$ is isolated from
       nontrivial solutions whose energies are lower than that of $\psi$, then
       $\psi$ is unstable.
\end{enumerate}
\end{proposition}

In~\cite[\S 4]{AK09}, it is proved that $\mathcal X$ forms a separatrix of the
dynamical system generated by \eqref{eq:1.6}--\eqref{eq:1.8} in the
whole of the energy space $H^1_0(\Omega)$ to divide its stable and
unstable sets. Moreover, it is also pointed out that $\mathcal X$ is
different from the so-called \emph{Nehari manifold} $\mathcal N := \{w
\in H^1_0(\Omega) \setminus \{0\} \colon \langle J'(w), w \rangle_{H^1_0(\Omega)} = 0\}$
and $\mathcal X \cap \mathcal N = \mathcal S$.

However, these stability criteria can not cover all situations. For
instance, in a thin annular domain case, it is known that least
energy solutions form a continuum in $H^1_0(\Omega)$ due to the symmetry
breaking of least energy solutions (see Coffman~\cite{Coffman} and
also~\cite{Li,Byeon}) and the invariance of the equation to rotations.
So one cannot apply Proposition \ref{P:criteria} to determine the
stability of such non-isolated least energy solutions to \eqref{eq:1.10},
\eqref{eq:1.11} in the sense of Definition \ref{D:stbl}
(cf.~see~\cite{INdAM13}). On the other hand, obviously, they are never
asymptotically stable.
 
The main purpose of the present paper is to prove the stability of
all (possibly non-isolated) asymptotic profiles of least energy. 
A main difficulty apparently stems from the lack of solitary of
asymptotic profiles. Behaviors of orbits near non-isolated stationary
points are treated in the study of dynamical systems, e.g., the center
manifold theory. In the current issue, the phase set $\mathcal X$ plays a crucial role to stabilize asymptotic profiles of least energy;
indeed, if one assigns the usual energy space $H^1_0(\Omega)$ as the
phase set instead of $\mathcal X$, all nontrivial stationary points of
the dynamical system generated by \eqref{eq:1.6}--\eqref{eq:1.8} are
saddle points of the Lyapunov energy $J(\cdot)$ and turn out to be
unstable. However, there are many unknown points regarding the
phase set $\mathcal X$, e.g., even the smoothness of $\mathcal X$ is
still unclear. So it seems difficult to directly apply the standard
approach to the dynamical system on $\mathcal X$. To overcome such a
difficulty, we shall turn our attention to the so-called \emph{\L
ojasiewicz-Simon inequality} (see~\cite{FeiSim00}), which is used to
investigate the convergence of solutions to non-isolated stationary
solutions for strongly nonlinear evolution equations including
degenerate and singular parabolic equations.

The main result of the present paper is stated as follows:
\begin{theorem}[Stability of asymptotic profiles of least energy]\label{T:MR1}
 Let $\phi > 0$ be a least energy solution of \eqref{eq:1.10},
 \eqref{eq:1.11}. Then $\phi$ is stable under the
 flow on $\mathcal X$ generated by solutions for
 \eqref{eq:1.6}--\eqref{eq:1.8} {\rm (}that is, $\phi$ is a stable asymptotic profile for FDE in the sense of Definition \ref{D:stbl}{\rm )}.
\end{theorem}

Here we remark that every least energy solution of \eqref{eq:1.10},
\eqref{eq:1.11} is sign-definite by strong maximum principle. Hence one can
assume the positivity of $\phi$ in $\Omega$ without any loss of
generality.

As mentioned above, our proof of Theorem \ref{T:MR1} will rely on the \L
ojasiewicz-Simon inequality (see~\cite{FeiSim00}).
The \L ojasiewicz-Simon inequality has been vigorously studied so far, and it is usually employed to prove the convergence of each solution for
nonlinear parabolic (and also damped wave) equations to a prescribed
(possibly non-isolated) stationary solution as $t \to \infty$ (and hence, the
$\omega$-limit set of each evolutionary solution turns out to be
singleton). More precisely, let $E : X \to \mathbb R$ be a ``smooth''
functional defined on a Banach space $X$ and let $\psi$ be a critical
point of $E$, i.e., $E'(\psi) = 0$ in the
dual space $X^*$, where $E' : X \to X^*$ denotes the Fr\'echet
derivative of $E$. Then an abstract form of the \L ojasiewicz-Simon
inequality is as follows (see,
e.g.,~\cite{Simon83,Jen98,HJ98,Har00,FeiSim00,HJ01,HJK03,Chill03,CHJ09,HJ11}):
there exist constants $\theta \in (0,1/2]$ and $\omega,\delta >0$ such that
$$
|E(v)-E(\psi)|^{1-\theta} \leq \omega \|E'(v)\|_{X^*} \quad \mbox{ for all
} v \in X \ \mbox{ satisfying } \ \|v-\psi\|_X < \delta
$$
(cf.~there are several variants with different choices of norms).
Here the constants $\theta$, $\omega$, $\delta$ may depend on the
choice of each critical point $\psi$ of the functional $E$. To
prove the convergence of a flow of a dissipative dynamical system along
with $E(\cdot)$ as a Lyapunov energy to a prescribed limit $\phi$, one
assigns $\phi$ to the critical point $\psi$ of the \L ojasiewicz-Simon
inequality, and then investigates the behavior of the flow for
sufficiently large time.
By contrast, to discuss the (Lyapunov) stability of a stationary point $\phi$ of
the system, the limit of each flow (emanating from a
neighborhood of $\phi$) is not prescribed. Here we focus on
the behavior of the flow near the initial time by assigning the target of
stability analysis (i.e., $\phi$) to
the critical point $\psi$ of the \L ojasiewicz-Simon inequality.

However, another difficulty then arises from the frame of stability
analysis. More precisely, in Definition \ref{D:stbl}, the notions of
stability are  formulated in the energy space $H^1_0(\Omega)$, whose
elements may not be uniformly bounded in $\Omega$. On the other hand,
due to the nonlinearity of FDE (see, e.g., Lemma \ref{L:1}), uniform estimates for solutions of
\eqref{eq:1.6}--\eqref{eq:1.8} will be required to investigate the
stability by using the \L ojasiewicz-Simon inequality, which is also
established in~\cite{FeiSim00} for \emph{uniformly bounded} functions in a small
neighborhood of each solution $\phi$ of \eqref{eq:1.10}, \eqref{eq:1.11}
with non-integer power $m > 1$. Therefore we need to
compensate the gap between the frame of stability analysis and
the validity of the argument based on the \L ojasiewicz-Simon
inequality. To this end, we shall develop a uniform extinction estimate for (possibly sign-changing) solutions of FDE by utilizing some results of~\cite{DKV91} and~\cite{DK92}.

Moreover, the \L ojasiewicz-Simon inequality will be also applied to
prove the \emph{instability} of asymptotic profiles for FDE. Let us consider
the annular domain
$$
A_N(a,b) := \left\{
x \in \mathbb R^N \colon a < |x| < b
\right\}
$$ 
with $0 < a < b < \infty$. As mentioned above, the positive radial
asymptotic profile for FDE does not take the least energy, provided that
the thickness $(b-a)/a$ of the annulus is sufficiently thin; thereby it is
beyond the scope of Proposition \ref{P:criteria}. One may
expect that the positive radial profile is \emph{unstable} (i.e., not
stable) in the sense of Definition \ref{D:stbl}. This conjecture was proved
only for the two dimensional case, $N = 2$, without providing any
quantitative information of the thickness of the annulus in~\cite{AK12}, where
the restriction on the space dimension $N$ and the lack of quantitative
information of the thickness arise from some technical difficulty of
spectral analysis of the corresponding linearized operator. The general
$N$-dimensional case has been left as an open question (cf.~it was
proved for general $N$ in~\cite{AK12} that the positive radial profile is
not asymptotically stable). In this paper, we shall also prove the
instability of the positive radial profile for general spacial dimension
$N$ and give an upper bound of the thickness of the annulus by applying
the \L ojasiewicz-Simon inequality.

\begin{theorem}[Instability of positive radial asymptotic profiles in
 thin annuli]\label{T:inst}
 Let $\Omega = A_N(a,b)$ and assume
\begin{equation}\label{hypo:inst}
\left( \frac{b}{a} \right)^{(N-3)_+} \left(\frac{b-a}{\pi a}\right)^2 
   < \frac{m-2}{N-1}. 
\end{equation}
Then the positive radial solution $\phi$ of \eqref{eq:1.10}, \eqref{eq:1.11} is unstable in the sense of Definition
 \ref{D:stbl}.  
\end{theorem}

This paper consists of five sections: In Section \ref{S:P}, we prepare several lemmas to be used in a proof of Theorem \ref{T:MR1}. Section \ref{S:M} is
devoted to proving Theorem \ref{T:MR1}. More precisely, we shall prove
the stability for all \emph{local minimizers of $J$ over $\mathcal X$}
(see \eqref{locmin} below for definition). Since every asymptotic
profile of least energy is a (global) minimizer of $J$ over $\mathcal
X$, Theorem \ref{T:MR1} will be also obtained as a special case. In
Section \ref{S:lm}, we discuss a couple of properties of local minimizers of $J$
over $\mathcal X$. In particular, we investigate the relation of (local)
minimizers of $J$ over $\mathcal X$ and those over the so-called \emph{Nehari
manifold} $\mathcal N$, which has been vigorously studied in variational
analysis of nonlinear elliptic equations. The final section is concerned
with the instability of positive radial asymptotic profiles in thin
annular domains.

\bigskip 

\noindent
{\bf Notation.} 
Let $u = u(x,t): \Omega \times [0,\infty) \to \mathbb R$ be a function with
space and time variables. Throughout the paper, for each $t \geq 0$ fixed,
we simply denote by $u(t)$ the function $u(\cdot,t) : \Omega \to \mathbb
R$ with only the space variable. We denote by $H^{-1}(\Omega)$ the
dual space of $H^1_0(\Omega)$. Moreover, $B_{H^1_0(\Omega)}(\phi ;
r)$ denotes the open ball in $H^1_0(\Omega)$ with radius $r > 0$
centered at $\phi$, i.e.,
$$
B_{H^1_0(\Omega)}(\phi ; r) := \left\{
w \in H^1_0(\Omega) \colon \|w - \phi\|_{H^1_0(\Omega)} < r
\right\} \quad \mbox{ for } \ r > 0.
$$
Furthermore, $C_m$ and $R(\cdot)$ stand for the best possible constant
of the Sobolev-Poincar\'e inequality \eqref{eq:1.12} below and the
corresponding Rayleigh quotient, respectively (see \S \ref{Ss:Prelim} for more details). For $T > 0$, $C_w([0,T];X)$ and $C_+([0,T];X)$ stand for the sets of weakly- and right- continuous functions on $[0,T]$ with values in a normed space $X$, respectively.

\section{Preliminaries and Lemmas}\label{S:P}

In this section, we collect preliminary facts and several lemmas.

\subsection{Preliminaries}\label{Ss:Prelim}
Let us start with recalling the definition of solutions.
\begin{definition}\label{D:sol}
 A function $u : \Omega \times (0,\infty) \to \mathbb R$ is said to be a
 \emph{solution} of \eqref{eq:1.1}--\eqref{eq:1.3}, if the following
 conditions hold true\/{\rm :}
\begin{itemize}
 \item \III $u \in L^\infty(0,T;H^1_0(\Omega))$ and 
       $|u|^{m-2}u \in W^{1,\infty}(0,T;H^{-1}(\Omega))$
       for any $T > 0$. \EEE 
 \item It holds that
       \begin{align*}
  \left\langle \partial_t \left(|u|^{m-2}u\right)(t), \phi \right\rangle_{H^1_0}
   + \int_\Omega \nabla u(x,t) \cdot \nabla \phi(x) \,\d x
   = 0\\
	\nonumber
	 \qquad \mbox{ for a.e.~} \ t \in (0, \infty)  
	 \mbox{ and } \ \phi \in H^1_0(\Omega),
       \end{align*}
       where $\langle \cdot, \cdot \rangle_{H^1_0}$ denotes the duality
       pairing between $H^1_0(\Omega)$ and its dual space
       $H^{-1}(\Omega)$.
 \item $u(\cdot, 0) = u_0$ a.e.~in $\Omega$.
\end{itemize}
\end{definition}
Solutions of \eqref{eq:1.6}--\eqref{eq:1.8} are also defined in an
analogous manner. The well-posedness of \eqref{eq:1.1}--\eqref{eq:1.3} in
the sense of Definition \ref{D:sol} is well known (see,
e.g.,~\cite{HB3},~\cite{Vazquez}). Hence the extinction time $t_* =
t_*(u_0)$ is uniquely determined for each initial data $u_0$. \III Moreover, one can also ensure that
\begin{align}
u \in C([0,T];L^m(\Omega)) \cap C_w([0,T];H^1_0(\Omega)) \cap C_+([0,T];H^1_0(\Omega)) \quad \mbox{ for any } \ T > 0, \label{apdx1}\\
 \partial_t (|u|^{m-2}u) \in C_+([0,T];H^{-1}(\Omega)) \quad \mbox{ for any } \ T > 0\label{apdx2}
\end{align}
(see Appendix for more details). \EEE 

Equation \eqref{eq:1.6} can be formulated as a generalized gradient flow
in $H^{-1}(\Omega)$ of the form,
\begin{equation}\label{ggf}
 \partial_s \left(|v|^{m-2}v\right)(s) = - J'(v(s)) \ \mbox{ in }
  H^{-1}(\Omega), \quad s > 0,
\end{equation}
where $J'$ stands for the Fr\'echet derivative of the energy functional
$J : H^1_0(\Omega) \to \mathbb R$.
Therefore the following energy inequalities hold true:
\begin{align}
 \frac{1}{m'} \frac{\d}{\d s} \|v(s)\|_{L^m(\Omega)}^m +
 \|v(s)\|_{H^1_0(\Omega)}^2 = \lambda_m \|v(s)\|_{L^m(\Omega)}^m,
  \label{eq:3.1}\\
 \mu_m \left\|
      \partial_s \left( |v|^{(m-2)/2}v \right) (s)
     \right\|_{L^2(\Omega)}^2
 + \frac \d{\d s} J(v(s)) \leq 0
\quad \mbox{ for a.e. } \, s > 0,
 \label{eq:3.2}
\end{align}
where $m'$ is the H\"older conjugate of $m$, i.e., $m' := m/(m-1)$ and
$\mu_m := 4/(mm') > 0$ (see, e.g.,~\cite{G:EnSol} for the
precise derivation of these energy inequalities). In particular, $s
\mapsto J(v(s))$ is non-increasing.

Define a Rayleigh quotient by
$$
R(w) := \frac{\|w\|_{H^1_0(\Omega)}}{\|w\|_{L^m(\Omega)}} \quad \mbox{
for } \ w \in H^1_0(\Omega) \setminus \{0\},
$$
associated with the Sobolev-Poincar\'e inequality
\begin{equation}\label{eq:1.12}
\|w\|_{L^m(\Omega)} \leq C_m \|w\|_{H^1_0(\Omega)} \quad \mbox{for }  w \in H^1_0(\Omega),
\end{equation}
provided that $m \in [1,2^*]$, with the best possible constant $C_m$
which is the supremum of $R(w)^{-1}$ over $w \in H^1_0(\Omega) \setminus
\{0\}$. Then the function $t \mapsto R(u(t))$ is non-increasing, and
hence, so is the function $s \mapsto R(v(s))$ (see, e.g.,~\cite{BH,
Kwong88-3, SavareVespri, AK09}).

Finally, we list up properties of the phase set $\mathcal X$ obtained
in~\cite{AK09} for later use.
\begin{proposition}[Properties of phase sets, cf.~\cite{AK09}]\label{P:X}
The phase set $\mathcal X$ satisfies the following properties\/{\rm :}
\begin{enumerate}
 \item If $v_0 \in \mathcal X$, then $v(s)$ lies on $\mathcal X$ for any
       $s \geq 0$.
 \item If $v_0 \in \mathcal X$, then there exist $\psi \in \mathcal S$
       and a sequence $s_n \to \infty$ such that $v(s_n) \to \psi$
       strongly in $H^1_0(\Omega)$.
 \item The set $\mathcal S$ is included in $\mathcal X$.
 \item The infimum of $J$ over $\mathcal X$ coincides with the least
       energy, i.e., the infimum of $J$ over $\mathcal S$. Moreover, if
       $w \in \mathcal X$ achieves the infimum, then $w$ is a least
       energy solution of \eqref{eq:1.10}, \eqref{eq:1.11}.
 \item For any $w \in \mathcal X$, it holds true that $t_*(w) = 1$.
 \item The set $\mathcal X$ is sequentially closed in the weak topology
       of $H^1_0(\Omega)$.
\end{enumerate} 
\end{proposition}

Proofs of (i)--(vi) can be found in~\cite[Propositions 5--8 and 10]{AK09}.

\subsection{Lemmas}

In this subsection, we shall develop several lemmas for later use. The following lemma provides a uniform estimate for (possibly sign-changing) solutions of the rescaled problem \eqref{eq:1.6}--\eqref{eq:1.8}. To prove this, we shall employ some results of DiBenedetto and Kwong~\cite{DK92} and DiBenedetto, Kwong and Vespri~\cite{DKV91}.

\begin{lemma}[Uniform estimate for rescaled solutions]\label{L:e:v-i}
 Assume \eqref{hypo}. Then there exists a constant $C > 0$ depending
 only on $N$, $m$ 
 such that for every $s_0 \in (0,\log 2)$ and $v_0 \in \mathcal X$, the
 unique solution $v = v(x,s)$ of \eqref{eq:1.6}--\eqref{eq:1.8}
 with the initial data $v_0$ satisfies
$$
\|v(s)\|_{L^\infty(\Omega)} \leq C \left( e^{s_0} - 1 \right)^{-\frac N
 {\kappa}} R(v_0)^{\frac{4m}{\kappa(m-2)}} \quad \mbox{
 for all } \ s \geq s_0
$$
with $\kappa := 2N - Nm + 2m > 0$ {\rm (}by \eqref{hypo}{\rm )}.
\end{lemma}

\begin{proof}
Let $u$ be a solution of \eqref{eq:1.1}--\eqref{eq:1.3} with an initial data $u_0 \in H^1_0(\Omega)$. Fix $T>0$, $R > 0$ and let $\Omega_* \subset \mathbb R^N$ be a smooth bounded domain such that
$$
\Omega \subset B_R \subset B_{4R} \subset \Omega_*,
$$
where $B_r := \{x \in \mathbb R^N \colon |x| < r\}$ for $r > 0$. Moreover, set a nonnegative function $\overline u_0 \in H^1_0(\Omega_*)$ by
$$
\overline u_0(x) = \begin{cases}
		    |u_0(x)| \quad &\mbox{ if } \ x \in \Omega,\\
		    0 &\mbox{ otherwise.}
		   \end{cases}
$$
Let $\overline u$ be the unique weak solution for \eqref{eq:1.1}--\eqref{eq:1.3} with $\Omega$, $\partial \Omega$ and $u_0$ replaced by $\Omega_*$, $\partial \Omega_*$ and $\overline u_0$, respectively. Then by the positivity result $\overline u > 0$ in $\Omega_* \times (0,T)$ due to~\cite{DKV91}, one particularly observes that
$$
\overline u > 0 \quad \mbox{ on } \ \partial \Omega \times (0,T).
$$
Hence by comparison principle,
\begin{equation}\label{cp}
u \leq \overline u \quad \mbox{ in } \ \Omega \times (0,T).
\end{equation}
A local $L^\infty$-estimate for \emph{nonnegative} solutions to FDE (see Theorem 3.1 of~\cite{DK92} with some change of notation, e.g., $u(x,t)$ and $m$ of~\cite{DK92} correspond to $u(x,t)^{m-1}$ and $1/(m-1)$, respectively, of our notation) yields that
$$
\sup_{x \in B_R} \overline u(x,t) \leq \gamma t^{-\frac N {\kappa}} \sup_{0 < \tau < t} \left( \int_{B_{2R}} \overline u(x,\tau)^m \, \d x \right)^{\frac 2 {\kappa}} + \gamma \left( \frac t {R^2} \right)^{\frac{1}{m-2}} \quad \mbox{ for } \ t \in (0,T)
$$
with some constant $\gamma = \gamma(N,m) > 0$ and $\kappa := 2N-Nm+2m > 0$ (by \eqref{hypo}). Here, by using a standard energy estimate for FDE, one can derive
$$
\sup_{t \geq 0} \|\overline u(t)\|_{L^m(\Omega_*)}^m
\leq \|\overline u_0\|_{L^m(\Omega_*)}^m
= \|u_0\|_{L^m(\Omega)}^m,
$$
which along with \eqref{cp} and the relation $\Omega \subset B_R$ implies
$$
\sup_{x \in \Omega} u(x,t)
\leq \sup_{x \in B_R} \overline u(x,t)
\leq \gamma t^{-\frac N {\kappa}} \|u_0\|_{L^m(\Omega)}^{\frac{2m}{\kappa}} + \gamma \left( \frac t {R^2} \right)^{\frac 1 {m-2}}
 \quad \mbox{ for } \ t \in (0,T).
$$
Since $\gamma$ and $\kappa$ are independent of $R$ and $T$, by letting $R, T \to \infty$, we conclude that 
$$
\sup_{x \in \Omega} u(x,t)
\leq \gamma t^{-\frac N {\kappa}} \|u_0\|_{L^m(\Omega)}^{\frac{2m}{\kappa}} \quad \mbox{ for all } \ t > 0.
$$
Repeating the preceding argument with $u$ and $u_0$ replaced by $-u$ and $-u_0$, respectively, we deduce that
$$
\sup_{x \in \Omega} |u(x,t)|
\leq \gamma t^{-\frac N {\kappa}} \|u_0\|_{L^m(\Omega)}^{\frac{2m}{\kappa}} \quad \mbox{ for all } \ t > 0.
$$
Furthermore, replace $u_0$ by $u(s)$ for $0 < s < t$ to get
\begin{equation}\label{loc-bdd-est}
\sup_{x \in \Omega} |u(x,t)|
\leq \gamma (t-s)^{-\frac N {\kappa}} \|u(s)\|_{L^m(\Omega)}^{\frac{2m}{\kappa}} \quad \mbox{ for all } \ 0 < s < t < \infty.
\end{equation}

In particular, let us set $u_0 = v_0 \in \mathcal X$. Then $u$ vanishes at $t_*(v_0) = 1$. Moreover, let $t_0 \in (0, 1/2)$ be given by
$$
s_0 = \log \left(\frac{1}{1-t_0}\right) > 0.
$$
As in~\cite[Lemma 6.1]{DKV91} (see also~\cite{SavareVespri}), substituting
$$
s = t - \dfrac{t_0}{1-t_0} (1-t) = \dfrac{t-t_0}{1-t_0} \in (0,t) \quad \mbox{ for } \ t \in (t_0,1)
$$
to \eqref{loc-bdd-est} and employing~\cite[Proposition 2]{AK09}, one can derive that
$$
\|u(t)\|_{L^\infty(\Omega)} \leq C_0(1-t)_+^{\frac 1 {m-2}}
\quad \mbox{ for all } \ t \geq t_0
$$
with a constant $C_0$ given by
$$
C_0 
= \tilde\gamma \left(\dfrac{t_0}{1 - t_0}\right)^{-\frac{N}{\kappa}}
 R(u_0)^{\frac{4m}{\kappa(m-2)}}
= \tilde \gamma \left( e^{s_0} - 1 \right)^{-\frac{N}{\kappa}}
 R(v_0)^{\frac{4m}{\kappa(m-2)}},
$$
where $\tilde\gamma$ is a constant depending only on $m$, $N$. 
By change of variables, $v(x,s) = (1 - t)^{-1/(m-2)} u(x,t)$ and $s = \log
 (1/(1-t))$, we find that
$$
\|v(s)\|_{L^\infty(\Omega)} \leq C_0
\quad \mbox{ for all } \ s \geq s_0.
$$
The proof is completed. 
\end{proof}

We next exhibit a couple of variational properties of the Rayleigh
quotient on the set $\mathcal X$.

\begin{lemma}[Rayleigh quotient on $\mathcal X$]\label{L:infR}
Assume \eqref{hypo}. It follows that
\begin{align*}
\inf_{w \in \mathcal X} R(w) = C_m^{-1} > 0, \quad
\inf_{w \in \mathcal X} \|w\|_{L^m(\Omega)} \geq (\lambda_m
 C_m^2)^{-1/(m-2)} > 0,\\
R(w) \leq (\lambda_m C_m^2)^{1/(m-2)} \|w\|_{H^1_0(\Omega)} \quad \mbox{ for all } \ w \in \mathcal X.
\end{align*}
In particular, $R(w) < \infty$ for all $w \in
 \mathcal X$. Moreover, $R(\cdot)$ is continuous on $\mathcal X$ in the
 strong topology of $H^1_0(\Omega)$.
\end{lemma}

\begin{proof}
 Let $w \in \mathcal X$ and recall that $t_*(w) = 1$
 (see Proposition \ref{P:X}). From the estimates from below and
 above for the extinction time $t_*(\cdot)$ (see~\cite[Corollary 1]{AK09}),
 it follows that
$$
\lambda_m \|w\|_{L^m(\Omega)}^{m-2} R(w)^{-2} \leq 1 \leq \lambda_m C_m^2
 \|w\|_{L^m(\Omega)}^{m-2}, 
$$
which yields that
$$
\dfrac 1 {\lambda_m C_m^2} \leq \|w\|_{L^m(\Omega)}^{m-2}
\quad \mbox{ and } \quad
\lambda_m \|w\|_{L^m(\Omega)}^{m-2} \leq R(w)^2.
$$
Hence we observe that $R(w) \geq C_m^{-1}$ and $\|w\|_{L^m(\Omega)} \geq
 (\lambda_m C_m^2)^{-1/(m-2)} > 0$ for all $w \in
 \mathcal X$. Moreover, it is known that $R(\psi) = C_m^{-1}$ for least energy solutions $\psi$ of \eqref{eq:1.10}, \eqref{eq:1.11} under $m < 2^*$. It follows that
$$
R(w) = \dfrac{\|w\|_{H^1_0(\Omega)}}{\|w\|_{L^m(\Omega)}} \leq (\lambda_m C_m^2)^{1/(m-2)} \|w\|_{H^1_0(\Omega)} \quad \mbox{ for all } \ w \in \mathcal X.
$$
Moreover, if $w_n \in \mathcal
 X$ and $w_n \to w$ strongly in $H^1_0(\Omega)$ (hence, $w \in \mathcal
 X$ by Proposition \ref{P:X}), then one can derive that $R(w_n)
 \to R(w)$, since $\|w_n\|_{L^m(\Omega)}$ and $\|w\|_{L^m(\Omega)}$ are
 not less than $(\lambda_m
 C_m^2)^{-1/(m-2)} > 0$.
\end{proof}

Moreover we have:
\begin{lemma}[Estimate for solutions on $\mathcal X$]\label{L:est1}
 Let $v_0 \in \mathcal X$ and let $v$ be the solution of
 \eqref{eq:1.6}, \eqref{eq:1.7} for the initial data $v_0$. Then it
 holds that
$$
\sup_{s \geq 0} \|v(s)\|_{H^1_0(\Omega)}^2 \leq 
2 J(v_0) + \dfrac{2 R(v_0)^{2m/(m-2)}}{m \lambda_m^{2/(m-2)}}.
$$
\end{lemma}

\begin{proof}
Note that $v(s)$ belongs to $\mathcal X$ for all $s \geq 0$. 
Hence, by the proof of Lemma \ref{L:infR}, 
$$
\|v(s)\|_{L^m(\Omega)}^{m-2} \leq \dfrac{R(v(s))^2}{\lambda_m} \quad \mbox{ for
 all } \ s \geq 0.
$$
Since $J(v(\cdot))$ and $R(v(\cdot))$ are nonincreasing, it follows that
$$
\dfrac 1 2 \|\nabla v(s)\|_{L^2(\Omega)}^2
= J(v(s)) + \dfrac{\lambda_m}m \|v(s)\|_{L^m(\Omega)}^m
\leq J(v_0) + \dfrac{R(v_0)^{2m/(m-2)}}{m \lambda_m^{2/(m-2)}},
$$
which completes the proof.
\end{proof}

We close this section with the continuous dependence of solutions to
\eqref{eq:1.6}--\eqref{eq:1.8} on data.
\begin{lemma}[Continuous dependence of solutions on data]\label{L:conti-dep}
 For $i = 1,2$, let $v_i$ be solutions to \eqref{eq:1.6}--\eqref{eq:1.8} with
 initial data $v_{0,i} \in H^1_0(\Omega)$. It then holds true that
 $$
 \left\| |v_1|^{m-2}v_1(s) - |v_2|^{m-2}v_2(s)
 \right\|_{H^{-1}(\Omega)}^2
 \leq  \left\| |v_{0,1}|^{m-2}v_{0,1} - |v_{0,2}|^{m-2}v_{0,2}
 \right\|_{H^{-1}(\Omega)}^2 e^{2 \lambda_m s}
 $$
 for all $s \geq 0$.
\end{lemma}

This lemma can be proved in a standard way; however, we give a proof for the
 convenience of the reader. 

\begin{proof} Subtract equations and test it
 by $(-\Delta)^{-1} (|v_1|^{m-2}v_1(s) - |v_2|^{m-2}v_2(s))$ to see that
\begin{align*}
\dfrac 1 2 \dfrac{\d}{\d s} \left\| |v_1|^{m-2}v_1(s) - |v_2|^{m-2}v_2(s)
 \right\|_{H^{-1}(\Omega)}^2
+ \int_\Omega \left( v_1(s) - v_2(s) \right) \left( |v_1|^{m-2}v_1(s)
 - |v_2|^{m-2}v_2(s) \right) \,\d x\\
= \lambda_m \left\| |v_1|^{m-2}v_1(s) - |v_2|^{m-2}v_2(s)
 \right\|_{H^{-1}(\Omega)}^2.
\end{align*}
By using the monotonicity of $w \mapsto |w|^{m-2}w$ and by applying
 Gronwall's inequality, we obtain the desired conclusion.
\end{proof}

\section{Proof of Theorem \ref{T:MR1}}\label{S:M}

This section is devoted to a proof of Theorem \ref{T:MR1}. 
We shall prove the stability for all \emph{local minimizers $\phi$ of
$J$ over $\mathcal X$}, i.e., $\phi$ satisfies
\begin{equation}\label{locmin}
J(\phi) = \inf \{J(w) \colon w \in \mathcal X \cap
B_{H^1_0(\Omega)}(\phi;r_0)\}
\end{equation}
for some $r_0 > 0$. Obviously, every least energy solution of
\eqref{eq:1.10}, \eqref{eq:1.11} is a global minimizer of $J$ over
$\mathcal X$, since the least energy is the minimum of $J$ over $\mathcal
X$ and $\mathcal S \subset \mathcal X$ (see Proposition \ref{P:X}); hence it
always satisfies \eqref{locmin} (with $r_0 = \infty$).
Moreover, we stress again that $\phi$ is not supposed to be isolated
even in the neighborhood $\mathcal X \cap
B_{H^1_0(\Omega)}(\phi;r_0)$. Hence there might be a sequence of (local)
minimizers $w_n \in \mathcal X \cap B_{H^1_0(\Omega)}(\phi;r_0)
\setminus \{\phi\}$ converging to $\phi$ in $H^1_0(\Omega)$.

Our result reads,
\begin{theorem}[Stability of local minimizers of $J$ over $\mathcal
 X$]\label{T:MR2}
 Let $\phi > 0$ be a local minimizer of $J$ over $\mathcal X$. 
 Then $\phi$ is stable under the
 flow on $\mathcal X$ generated by solutions for
 \eqref{eq:1.6}--\eqref{eq:1.8}. Hence, in particular, Theorem
 \ref{T:MR1} holds true.
\end{theorem}

One of most crucial points of a proof for Theorem \ref{T:MR2} is how to
control the distance between $\phi$ and the solution $v(s)$ of
\eqref{eq:1.6}--\eqref{eq:1.8} emanating from a small neighborhood of
$\phi$. Here we first exhibit a strategy based on the \L
ojasiewicz-Simon inequality to estimate the distance between $\phi$ and
$v(s)$ before proceeding to a proof.

Let $\phi$ be a local minimizer of $J$ over $\mathcal X$ and let $r_0 >
0$ be such that \eqref{locmin} is satisfied. 
Since every local minimizer of $J$ over $\mathcal X$ is a sign-definite
(nontrivial) solution of \eqref{eq:1.10}, \eqref{eq:1.11} (see Proposition
\ref{P:LM} below), we can assume $\phi \geq 0$ without any loss of
generality. Moreover, by strong maximum principle and elliptic
regularity, one can assure that
\begin{equation}\label{smp}
0 < \phi(x) < L_\phi := \|\phi\|_{L^\infty(\Omega)} + 1 \ \mbox{ for all
 } \ x \in \Omega
\quad \mbox{ and } \quad
\partial_\nu \phi < 0 \ \mbox{ on } \partial \Omega.
\end{equation}
Then the following Feireisl-Simondon version (see~\cite{FeiSim00}) of
the \L ojasiewicz-Simon inequality holds true: 
\begin{lemma}[\L ojasiewicz-Simon inequality~\cite{FeiSim00}]\label{L:LS}
For any $L > L_\phi$, there exist
 constants $\theta \in (0,1/2]$, $\omega , \delta_0 > 0$ such that
\begin{align}\label {LS}
 \left| J(w) - J(\phi) \right|^{1-\theta}
 \leq \omega \left\| J'(w) \right\|_{H^{-1}(\Omega)},
\end{align}
whenever $w \in H^1_0(\Omega)$ satisfies $|w(x)| \leq L$ for a.e.~$x \in
\Omega$ and $\|w - \phi\|_{H^1_0(\Omega)} < \delta_0$.
\end{lemma}

This lemma follows from Proposition 6.1 of~\cite{FeiSim00}, where the \L ojasiewicz-Simon inequality is established for some functional associated with the operator $v \mapsto -\Delta v + F(v)$, by introducing a function  $F \in C^1(\mathbb R) \cap W^{1,\infty}(\mathbb R)$ satisfying $F(s) = - \lambda_m |s|^{m-2}s$ for all $s \in [-M,M]$ with $M := L+1$, $F \in C^\infty(0,M)$ and
$$
|F^{(n)}(s)| \leq \dfrac{r^n n!}{s^n} \quad \mbox{ for all } \ s \in (0,M), \ n \in \mathbb N
$$
for some $r > 0$ (cf.~see also \S 5 of~\cite{FeiSim00}). Furthermore, we remark that the positivity (or negativity) of $\phi$ is essentially required; however, the sign of $w$ is not specified in the proof of Proposition 6.1 of~\cite{FeiSim00}.

Throughout the rest of this section, let $s_0 \in (0,\log 2)$ be fixed.
By Lemma \ref{L:infR}, one can take $C_1 > 0$ such that
$$
R(v_0) \leq (\lambda_m C_m^2)^{1/(m-2)} \|\nabla v_0\|_{L^2(\Omega)}
\leq C_1 \quad \mbox{ for all } \ v_0 \in B_{H^1_0(\Omega)}(\phi
; r_0) \cap \mathcal X.
$$
By Lemma \ref{L:e:v-i}, one can take a constant $L =
L(s_0, C_1, N, m)> 0$ such that for any $v_0 \in
B_{H^1_0(\Omega)}(\phi;r_0) \cap \mathcal X$, the unique solution $v = v(x,s)$ of \eqref{eq:1.6}--\eqref{eq:1.8} satisfies
\begin{equation}\label{UB2}
\|v(s)\|_{L^\infty(\Omega)} \leq L \quad \mbox{ for all } \ s \geq s_0.
\end{equation} 
Here, we particularly took $L$ larger than $L_\phi$.
Then thanks to the \L ojasiewicz-Simon inequality (see Lemma
\ref{L:LS}), there exist constants $\theta \in (0,1/2]$, $\omega,
\delta_0 > 0$ such that for any $v_0 \in B_{H^1_0(\Omega)}(\phi;\delta_0
\wedge r_0) \cap \mathcal X$, the solution $v = v(x,s)$ of
\eqref{eq:1.6}--\eqref{eq:1.8} with the initial data $v_0$ satisfies 
\begin{equation}\label{LS2}
 \left( J(v(s)) - J(\phi) \right)^{1-\theta}
\leq \omega \|J'(v(s))\|_{H^{-1}(\Omega)},
\end{equation}
whenever $\|v(s)-\phi\|_{H^1_0(\Omega)} < \delta_0 \wedge r_0$ and $s \geq s_0$
(hence, \eqref{UB2} is satisfied).
Here we used the fact by \eqref{locmin} that $J(v(s)) - J(\phi) \geq 0$
whenever $v(s) \in B_{H^1_0(\Omega)}(\phi;r_0)$, since $v(s) \in
\mathcal X$ for all $s \geq 0$ (see Proposition \ref{P:X}).

Let $\delta,\delta'$ be real numbers such that
$$
0 < \delta' < \delta < \delta_0 \wedge r_0.
$$
Take any
$$
v_0 \in \mathcal X \cap B_{H^1_0(\Omega)}(\phi ; \delta')
$$
and denote by $v = v(x,s)$ the solution of \eqref{eq:1.6}--\eqref{eq:1.8}
with the initial data $v_0$. Then since $v$ belongs to
 \III $C_+([0,\infty) ; H^1_0(\Omega))$, \EEE one can take $s_{\delta} > 0$ such that
$$
v(s) \in B_{H^1_0(\Omega)}(\phi ; \delta) \quad \mbox{ for all } \ s \in
[0,s_{\delta}).
$$
Furthermore, let us recall that $v(s) \in \mathcal X$ for any $s \geq
0$ and \eqref{UB2} is satisfied. Moreover, suppose that
$$
s_0 < s_\delta.\eqno{\mbox{(A1)}}
$$
Then \eqref{LS2} holds true for all $s \in [s_0,s_\delta)$. 

Define
$$
H(s) := \left( J(v(s)) - J(\phi) \right)^\theta \geq 0
\quad \mbox{ for } \ s \in [0,s_\delta)
$$
and suppose that
$$
J(v(s)) - J(\phi) > 0 \quad \mbox{ for all } \ s \in [0, s_\delta).
\eqno{\mbox{(A2)}}
$$
Then we see that
\begin{align*}
 - \dfrac{\d}{\d s} H(s)
&= - \theta \left( J(v(s)) - J(\phi) \right)^{\theta - 1}
\dfrac{\d}{\d s} J(v(s))
\\
&\geq \mu_m \theta \left( J(v(s)) - J(\phi) \right)^{\theta - 1}
\left\| \partial_s \left( \vmh \right) (s)
 \right\|_{L^2(\Omega)}^2
\quad \mbox{ for a.e. } \, s \in (0,s_\delta).
\end{align*}
Here the last inequality follows from the energy inequality
\eqref{eq:3.2}.

Now, we claim that
\begin{lemma}\label{L:1} It holds that
$$
\left\| \partial_s \left( \vm \right)(s)\right\|_{L^2(\Omega)}^2
\leq \kappa_m \|v(s)\|_{L^\infty(\Omega)}^{m-2}
\left\| \partial_s \left( \vmh \right) (s) \right\|_{L^2(\Omega)}^2
\quad \mbox{ for a.e. } \, s > s_0
$$
with $\kappa_m := 4(m-1)^2/m^2 > 0$.
\end{lemma}

\begin{proof}
Set $\gamma(\sigma) := \sigr$ for $\sigma \in \mathbb R$ and determine $r > 1$ such that
$$
\gamma \left( \sigmh \right) = 
 \sigmcomp = \sigm \quad \mbox{
 for } \ \sigma \in \mathbb R.
$$
Then $r = (3m-2)/m > 1$. Hence
\begin{align*}
 \int_\Omega \left| \partial_s \left( \vm  \right)\right|^2 \, \d x
&= \int_\Omega \left| \partial_s  \gamma\left( \vmh \right)\right|^2 \, \d x\\
&= (r-1)^2 \int_\Omega \left|
\left| \vmh \right|^{r-2} \partial_s \left( \vmh \right)
\right|^2 \, \d x\\
&\leq (r-1)^2 \|v(s)\|_{L^\infty(\Omega)}^{m-2}
\left\|
\partial_s \left( \vmh \right)(s)
\right\|_{L^2(\Omega)}^2,
\end{align*}
which completes the proof.
\end{proof}

Recalling that $v(s) \neq 0$ by $0 \not\in\mathcal X$, we find that
\begin{align*}
- \dfrac{\d}{\d s} H(s)
&\geq \dfrac{\mu_m \theta}{\kappa_m} \|v(s)\|_{L^\infty(\Omega)}^{-(m-2)}
 \left( J(v(s)) - J(\phi) \right)^{\theta - 1}
\left\| \partial_s \left( \vm \right)(s) \right\|_{L^2(\Omega)}^2
\end{align*}
for a.e.~$s \in (s_0,s_\delta)$.
Since \eqref{LS2} holds for all $s \in [s_0,s_\delta)$, it follows that
\begin{align}
 - \dfrac{\d}{\d s} H(s)
 &\geq
\dfrac{\mu_m \theta}{\kappa_m\omega} \|v(s)\|_{L^\infty(\Omega)}^{-(m-2)}
\|J'(v(s))\|_{H^{-1}(\Omega)}^{-1}
\left\| \partial_s \left( \vm \right)(s)\right\|_{L^2(\Omega)}^2
\nonumber\\
&\geq
\dfrac{\mu_m \theta}{\kappa_m\omega C_2^2} \left( \sup_{s \geq s_0}
 \|v(s)\|_{L^\infty(\Omega)} \right)^{-(m-2)}
\left\| \partial_s \left( \vm \right)(s)\right\|_{H^{-1}(\Omega)}
\label{1}
\end{align}
for a.e.~$s \in (s_0,s_{\delta})$, by noting that
$$
\left\| \partial_s \left( \vm \right)(s)\right\|_{L^2(\Omega)} \geq
C_2^{-1} \left\| \partial_s \left( \vm \right)(s)\right\|_{H^{-1}(\Omega)}
\stackrel{\eqref{ggf}}{=} C_2^{-1} \left\| J'(v(s)) \right\|_{H^{-1}(\Omega)}
$$
with the best possible constant $C_2 > 0$ of \eqref{eq:1.12} with $m =
2$. Thus we obtain
\begin{align}
\lefteqn{
 \left\| \vm (s) - \phi^{m-1} \right\|_{H^{-1}(\Omega)}
}\nonumber\\
&\leq
 \left\| \vm (s) - \vm (s_0) \right\|_{H^{-1}(\Omega)}
+  \left\| \vm (s_0) - \phi^{m-1} \right\|_{H^{-1}(\Omega)}
\nonumber\\
&\leq
 \int^s_{s_0} \left\| \partial_\sigma
 \left( \vm \right)(\sigma)\right\|_{H^{-1}(\Omega)} \, \d \sigma
+  \left\| \vm (s_0) - \phi^{m-1} \right\|_{H^{-1}(\Omega)}
\nonumber\\
&\stackrel{\eqref{1}}{\leq}
 - \dfrac{\kappa_m\omega C_2^2}{\mu_m \theta} 
 \left( \sup_{s \geq s_0} \|v(s)\|_{L^\infty(\Omega)}^{m-2} \right)
 \left( H(s) - H(s_0) \right)
+  \left\| \vm (s_0) - \phi^{m-1} \right\|_{H^{-1}(\Omega)}
\nonumber\\
&\leq
\dfrac{\kappa_m\omega C_2^2}{\mu_m \theta} 
 \left( \sup_{s \geq s_0} \|v(s)\|_{L^\infty(\Omega)}^{m-2} \right) 
H(s_0)
+  \left\| \vm (s_0) - \phi^{m-1} \right\|_{H^{-1}(\Omega)}
\label{2}
\end{align}
for all $s \in [s_0,s_\delta]$.

Now, we are ready to prove the stability of $\phi$. 
\begin{proof}[Proof of Theorem \ref{T:MR2}]
Suppose on the contrary that there exists $\vep_0
> 0$ such that for all $n \in \mathbb N$, there exist solutions $v_n =
v_n(x,t)$ of \eqref{eq:1.6}--\eqref{eq:1.8} satisfying
$$
v_n(0) \in \mathcal X, \quad \|v_n(0) - \phi\|_{H^1_0(\Omega)} < \dfrac
1 n, \quad \sup_{s \geq 0} \|v_n(s) - \phi\|_{H^1_0(\Omega)} \geq
\vep_0
$$
(see Definition \ref{D:stbl}). Here we note that
\begin{equation}\label{c:v0}
 \vnm (0) \to \phi^{m-1} \quad \mbox{ strongly in }
  H^{-1}(\Omega),
\end{equation}
since the operator $w \mapsto |w|^{m-2}w$ is continuous from $L^m(\Omega)$ to $L^{m'}(\Omega)$ and $H^1_0(\Omega)$ (resp., $L^{m'}(\Omega)$) is continuously embedded in $L^m(\Omega)$ (resp., $H^{-1}(\Omega)$). Set $\vep_1 := (\vep_0 \wedge \delta_0 \wedge r_0)/2 > 0$. Then from the \III right-continuity \EEE of $s \mapsto v_n(s)$ in the strong topology of $H^1_0(\Omega)$ on $[0,\infty)$, for each $n > \vep_1^{-1}$, one can take $s_n > 0$ such that
\begin{equation}\label{ass_cont}
\|v_n(s_n) - \phi\|_{H^1_0(\Omega)} = \vep_1 \quad \mbox{ and } \quad
\|v_n(s) - \phi \|_{H^1_0(\Omega)} < \vep_1 \ \mbox{ for all } s \in [0,s_n).
\end{equation}
 \III Indeed, by the right-continuity of $s \mapsto v_n(s)$ in the strong topology of $H^1_0(\Omega)$, we infer that
$$
s_n := \inf \{s > 0 \colon \|v_n(s)-\phi\|_{H^1_0(\Omega)} \geq \vep_1 \} \in (0,\infty)
$$
(then there exists a sequence $\sigma_k \searrow s_n$ such that $\|v_n(\sigma_k) - \phi\|_{H^1_0(\Omega)} \geq \vep_1$ for all $k$) and
$$
\|v_n(s_n) - \phi\|_{H^1_0(\Omega)} = \lim_{\sigma_k \searrow s_n} \|v_n(\sigma_k) - \phi\|_{H^1_0(\Omega)} \geq \vep_1.
$$
Moreover, $\|v_n(s)-\phi\|_{H^1_0(\Omega)} < \vep_1$ for all $s \in [0,s_n)$. Since $s \mapsto v_n(s)$ is continuous in the weak topology of $H^1_0(\Omega)$, it holds that
$$
\vep_1 \geq \limsup_{s \nearrow s_n} \|v_n(s)-\phi\|_{H^1_0(\Omega)} \geq \liminf_{s \nearrow s_n} \|v_n(s)-\phi\|_{H^1_0(\Omega)} \geq \|v_n(s_n) - \phi\|_{H^1_0(\Omega)}.
$$
Thus we obtain $\|v_n(s_n)-\phi\|_{H^1_0(\Omega)} = \vep_1$. \EEE 
In particular, $(v_n(s_n))$ is bounded in $H^1_0(\Omega)$.

In order to apply \eqref{2}, we shall check the assumptions (A1) and
(A2). We first claim that
\begin{lemma}[Check of (A1)]\label{L34}
 It holds that $s_n \to \infty$ as $n \to \infty$. In particular, $s_n >
 s_0 \in (0,\log 2)$ for $n \in \mathbb N$ large enough.
Moreover, $v_n(s) \to \phi$ strongly in $H^1_0(\Omega)$ at each $s \geq 0$.
\end{lemma}

\begin{proof}
Indeed, suppose on the contrary that a subsequence $(s_{n'})$ of $(s_n)$ is
 bounded, i.e., $S := \sup \{s_{n'} \colon n' \in \mathbb N \} <
 \infty$. From now on, we simply write $n$ instead of $n'$. Since
 $v_{n}(0) \to \phi$ strongly in $H^1_0(\Omega)$ and $\phi$ is a
 stationary solution, by Lemma \ref{L:conti-dep}, one can prove that
 $\vnm \to \phi^{m-1}$ strongly in $C([0,S];H^{-1}(\Omega))$. Moreover,
 by Tartar's inequality,
$$
\omega_m |a-b|^m \leq \left(|a|^{m-2}a - |b|^{m-2}b\right) \left(a-b\right)
\quad \mbox{ for all } \ a,b \in \mathbb R,
$$
for some constant $\omega_m > 0$, we find that
\begin{align}
\omega_m  \left\| v_n(s) - \phi \right\|_{L^m(\Omega)}^m
&\leq \int_\Omega \left( \vnm (s) - \phi^{m-1} \right) \left(v_n(s) - \phi \right) \, \d x\nonumber\\
&\leq \left\| \vnm (s) - \phi^{m-1}
 \right\|_{H^{-1}(\Omega)}
\left\| v_n(s) - \phi \right\|_{H^1_0(\Omega)} \label{tartar}.
\end{align}
From the boundedness of $(v_n)$ in
 $L^\infty(0,S;H^1_0(\Omega))$ (by Lemma \ref{L:est1} and the boundedness of
 $J(v_n(0))$ and $R(v_n(0))$)
 along with the convergence of $\vnm$ in $C([0,S];H^{-1}(\Omega))$, it
 follows that
\begin{equation}\label{c:vn-Lm:C}
v_n \to \phi \quad \mbox{ strongly in } C([0,S];L^m(\Omega)).
\end{equation}
By subtraction of equations, we have
$$
\partial_s \left( \vnm \right)(s)
- \Delta (v_n(s) - \phi) 
= \lambda_m \left( \vnm (s) - \phi^{m-1} \right)
\ \mbox{ in } H^{-1}(\Omega), \quad s > 0.
$$
Let us formally test it by $\partial_s v_n(s) = \partial_s
 (v_n(s)-\phi)$ to get
\begin{align*}
\lefteqn{
 \mu_m \left\|\partial_s
 \left(\vnmh \right)(s)\right\|_{L^2(\Omega)}^2
 + \dfrac 1 2 \dfrac{\d}{\d s} \|\nabla v_n(s) - \nabla \phi\|_{L^2(\Omega)}^2
}\\
 &\leq \lambda_m \int_\Omega
\left( \vnm (s) - \phi^{m-1} \right) \partial_s v_n(s) \,
 \d x\\
 &= \dfrac{\d}{\d s}
 \left(
 \dfrac{\lambda_m}m \|v_n(s)\|_{L^m(\Omega)}^m
 - \lambda_m \int_\Omega \phi^{m-1} \; v_n(s) \, \d x
 \right).
\end{align*}
The integration of both sides over $(0,s)$ leads us to see that
\begin{align*}
\lefteqn{
 \mu_m \int^s_0 \left\|\partial_\sigma
 \left( \vnmh \right)(\sigma)\right\|_{L^2(\Omega)}^2 \,\d \sigma
 + \dfrac 1 2 \|\nabla v_n(s) - \nabla \phi\|_{L^2(\Omega)}^2
}\\
 &\leq \dfrac 1 2 \|\nabla v_n(0) - \nabla \phi\|_{L^2(\Omega)}^2
 + \dfrac{\lambda_m}m \|v_n(s)\|_{L^m(\Omega)}^m
 - \dfrac{\lambda_m}m \|v_n(0)\|_{L^m(\Omega)}^m\\
&\quad - \lambda_m \int_\Omega \phi^{m-1}  \left( v_n(s) -
 v_n(0) \right) \, \d x,
 \end{align*}
which can be rigorously derived as in~\cite{G:EnSol}. Thus by
 virtue of \eqref{c:vn-Lm:C} one obtains
\begin{align*}
\lefteqn{
\dfrac 1 2 \sup_{s \in [0,S]} \|\nabla v_n(s) - \nabla \phi\|_{L^2(\Omega)}^2
}\\
 &\leq
\dfrac 1 2 \|\nabla v_n(0) - \nabla \phi\|_{L^2(\Omega)}^2
 + \dfrac{\lambda_m}m \sup_{s \in [0,S]} \left|
 \|v_n(s)\|_{L^m(\Omega)}^m - \|v_n(0)\|_{L^m(\Omega)}^m \right|\\
&\quad + \lambda_m \|\phi\|_{L^m(\Omega)}^{m-1} 
\sup_{s \in [0,S]}\| v_n(s) - v_n(0) \|_{L^m(\Omega)} \to 0.
\end{align*}
Therefore $v_n \to \phi$ strongly in $L^\infty(0,S;H^1_0(\Omega))$; in
 particular, we see that
$$
\|v_n(s_n) - \phi\|_{H^1_0(\Omega)} \leq
\sup_{s \in [0,S]} \|v_n(s) - \phi\|_{H^1_0(\Omega)} \to 0,
$$
which contradicts the fact that $\|v_n(s_n) -
\phi\|_{H^1_0(\Omega)} = \vep_1 > 0$. Hence $s_n$ diverges to $\infty$.

Moreover, repeating the argument above, one can also verify that
\begin{equation}\label{vn1-conv}
\sup_{s \in [0,S]} \| v_{n}(s) - \phi \|_{H^1_0(\Omega)} \to 0
\quad \mbox{ for any fixed } \ S > 0. 
\end{equation}
Thus we have proved the lemma.
\end{proof}

We next see that 
\begin{lemma}[Check of (A2)]\label{L32}
It holds that $J(v_n(s)) - J(\phi) > 0$ for all $s \in [0,s_n)$.
\end{lemma}
\begin{proof}
 Suppose on the contrary that $J(v_n(s_{0,n})) = J(\phi)$ for some
 $s_{0,n} \in [0,s_n)$. Then by \eqref{eq:3.2},
$$
\mu_m \int^{s_n}_{s_{0,n}} \left\|\partial_\sigma \left( \vnmh
 \right) (\sigma)\right\|_{L^2(\Omega)}^2 \, \d \sigma + J(v_n(s_n))
\leq J(v_n(s_{0,n})) = J(\phi),
$$
which along with the fact by \eqref{locmin} that $J(v_n(s_n)) \geq
 J(\phi)$ implies $\partial_s ( \vnmh ) \equiv 0$ a.e.~on $\Omega
 \times (s_{0,n}, s_n)$. Hence $v_n(s) =
 v_n(s_{0,n})$ for all $s \in [s_{0,n}, s_n]$. On the other hand,
 from the fact that $s_{0,n} < s_n$, one has $\|v_n(s_{0,n}) -
 \phi\|_{H^1_0(\Omega)} < \vep_1$. Combining these facts, we
 particularly obtain $\|v_n(s_n) - \phi\|_{H^1_0(\Omega)} < \vep_1$,
 which is a contradiction to the definition of $s_n$. Thus $J(v_n(s)) -
 J(\phi) > 0$ for all $s \in [0,s_n)$.
\end{proof}


Since $v_n(0) \in B_{H^1_0(\Omega)}(\phi;r_0) \cap \mathcal X$, by
\eqref{UB2} we see that
\begin{equation}\label{e:vn-i}
\sup_{s \geq s_0} \|v_n(s)\|_{L^\infty(\Omega)} \leq L
\quad \mbox{ for all } \ n \in \mathbb N.
\end{equation}
By taking $n \in \mathbb N$ so large that $s_n > s_0$ (see Lemma
\ref{L34}) and using Lemma \ref{L32}, one can employ \eqref{2} to obtain
\begin{align*}
 \left\| \vnm (s_n) - \phi^{m-1} \right\|_{H^{-1}(\Omega)}
&\leq
\dfrac{\kappa_m\omega C_2^2}{\mu_m \theta} 
 \left( \sup_{s \geq s_0} \|v_n(s)\|_{L^\infty(\Omega)}^{m-2} \right) 
 \left( J(v_n(s_0)) - J(\phi) \right)^{\theta}\\
&\quad +  \left\| \vnm(s_0) - \phi^{m-1}
 \right\|_{H^{-1}(\Omega)},
\end{align*}
which together with \eqref{e:vn-i} gives
\begin{align*}
 \left\| \vnm(s_n) - \phi^{m-1} \right\|_{H^{-1}(\Omega)}
\leq
 C\left( J(v_n(s_0)) - J(\phi) \right)^{\theta}
+  \left\| \vnm(s_0) - \phi^{m-1}
 \right\|_{H^{-1}(\Omega)}
\end{align*}
for some constant $C \geq 0$ independent of $n$. Hence, by 
\eqref{vn1-conv}, we deduce that
$$
\left\| \vnm(s_n) - \phi^{m-1} \right\|_{H^{-1}(\Omega)}
\to 0.
$$
As in \eqref{tartar} by Tartar's inequality, it follows that
\begin{equation}\label{c:vn:Lm}
v_n(s_n) \to \phi \quad \mbox{ strongly in } L^m(\Omega).
\end{equation}
Since $(v_n(s_n))$ is bounded in $H^1_0(\Omega)$ by \eqref{ass_cont},
up to a subsequence, $v_n(s_n) \to \phi$ weakly in $H^1_0(\Omega)$.

Furthermore, we deduce that \III 
\begin{align*}
 \dfrac 1 2 \|v_n(s_n)\|_{H^1_0(\Omega)}^2
 &= J(v_n(s_n)) + \dfrac {\lambda_m} m \|v_n(s_n)\|_{L^m(\Omega)}^m\\
 &\leq J(v_n(0)) + \dfrac {\lambda_m} m \|v_n(s_n)\|_{L^m(\Omega)}^m\\
 &\stackrel{\eqref{c:vn:Lm}}{\to} J(\phi) + \dfrac {\lambda_m} m
 \|\phi\|_{L^m(\Omega)}^m
 = \dfrac 1 2 \|\phi\|_{H^1_0(\Omega)}^2
\end{align*}
from the fact that $\|v_n(0)-\phi\|_{H^1_0(\Omega)} <
 1/n$ as well as the non-increase of the energy $J(v_n(\cdot))$. \EEE 
Due to the uniform convexity of $H^1_0(\Omega)$, we also obtain
$$
v_n (s_n) \to \phi \quad \mbox{ strongly in } H^1_0(\Omega).
$$
However, it contradicts the definition of $s_n$, i.e.,
$\|v_n(s_n)-\phi\|_{H^1_0(\Omega)} = \vep_1 > 0$.
Consequently, we conclude that $\phi$ is stable.
\end{proof}

\section{Local minimizers of $J$ over $\mathcal X$}\label{S:lm}

In this section, we are concerned with \emph{local
minimizers of $J$ over the set $\mathcal X$}.
Let us start with the following proposition, which was already used in
Section \ref{S:M}.
\begin{proposition}\label{P:LM}
 Let $\phi$ satisfy \eqref{locmin}. Then $\phi$ is a positive or
 negative solution of \eqref{eq:1.10}, \eqref{eq:1.11}.
\end{proposition}

\begin{proof}
 Let $v = v(x,s)$ be the solution of \eqref{eq:1.6}--\eqref{eq:1.8} with
 the initial data $v(0) = \phi$. Due to the \III right-continuity \EEE of $s \mapsto v(s)$
 in $H^1_0(\Omega)$, one can take $s_* \in (0,\infty]$ such that $v(s)
 \in B_{H^1_0(\Omega)}(\phi;r_0)$ for all $s \in [0,s_*)$.
 Then from \eqref{eq:3.2} along with the fact that
$$
J(\phi) = \inf \{J(w) \colon w \in \mathcal X \cap
B_{H^1_0(\Omega)}(\phi;r_0)\} \leq J(v(s)) \leq J(v(0)) = J(\phi)
\quad \mbox{ for all } \ s \in [0,s_*),
$$
it follows that
$$
\int^s_0 \left\| \partial_\sigma \left( |v|^{(m-2)/2}v
 \right)(\sigma)\right\|_{L^2(\Omega)}^2 \, \d \sigma = 0 \quad \mbox{ for all } \ s
 \in [0,s_*),
$$
which implies $v(s) = v(0) = \phi$ for all $s \in [0,s_*]$. Hence we have $s_* =
 \infty$ and $v(s) \equiv \phi$. Therefore $\phi$ solves
 \eqref{eq:1.10}, \eqref{eq:1.11}. We next prove the positivity (or
 negativity) of $\phi$. Suppose on the contrary that $\phi$ is
 sign-changing. Then let $D \subsetneq \Omega$ be a nodal domain of $\phi$,
 a connected component of the set $\{x \in \Omega \colon \phi(x) \neq
 0\}$. As in~\cite[Proof of Theorem 3]{AK09}, one can define
$$
\phi_\mu(x) := \begin{cases}
		\mu \phi(x) &\mbox{ if } x \in D,\\
		\phi(x) &\mbox{ if } x \not\in D
	       \end{cases}
$$
and observe that $J(c \phi_\mu) < J(\phi)$ for any $\mu \geq 0$, $\mu
 \neq 1$ and any $c \geq 0$. Hence put $v_{0,\mu} :=
 t_*(\phi_\mu)^{-1/(m-2)} \phi_\mu \in \mathcal X$. Then
$$
 J(v_{0,\mu}) < J(\phi) \quad \mbox{ for any } \ \mu \geq 0, \ \neq 1.
$$
On the other hand, from the continuity of $t_*
 : H^1_0(\Omega) \to [0,\infty)$ (see~\cite[Proposition 4]{AK09}) and
 the fact that $t_*(\phi) = 1$ by $\phi \in \mathcal X$, one deduces
 that
$$
v_{0,\mu} \to \phi \quad \mbox{ strongly in } H^1_0(\Omega),
$$ 
whence $v_{0,\mu}$ belongs to $B_{H^1_0(\Omega)}(\phi;r_0)$ for
 $\mu$ sufficiently close to $1$. However, these facts contradict the
 local minimality of $J$ at $\phi$ over $\mathcal X$. Therefore $\phi$
 turns out to be nonnegative or nonpositive. Finally, by strong maximum principle, $\phi$ is positive or negative in $\Omega$.
\end{proof}

\begin{remark}
{\rm
In the annular domain case, $\Omega := \{x \in \mathbb R^N \colon 
a < |x| < b\}$ for $0 < a < b < \infty$, as in~\cite[Proposition
 5.3]{AK12}, one may also prove that every local minimizer of $J$ over
 $\mathcal X$ is not radially symmetric under some quantitative
 assumption on the thickness of the annulus. Indeed, suppose on the
 contrary that a local minimizer $\phi$ is radially symmetric. Then one
 can construct $v_{0,\mu} \in \mathcal X$ such that $J(v_{0,\mu}) <
 J(\phi)$ for any $0 < \mu \ll 1$ and $v_{0,\mu} \to \phi$ strongly in
 $H^1_0(\Omega)$ as $\mu \to 0$, if $a$ and $b$ satisfy
 $(b/a)^{(N-3)_+}((b-a)/(\pi a))^2 < (m-2)/(N-1)$. Hence these facts
 yield a contradiction.
}
\end{remark}

Let us next discuss the relation of local minimizers of $J$ over the
so-called \emph{Nehari manifold},
$$
\mathcal N := \{w \in H^1_0(\Omega) \setminus \{0\} \colon \|\nabla
 w\|_{L^2(\Omega)}^2 = \lambda_m \|w\|_{L^m(\Omega)}^m \},
$$
and those over $\mathcal X$. Emden-Fowler equation
\eqref{eq:1.10}, \eqref{eq:1.11} has been well studied in variational
analysis, where nontrivial solutions are often characterized as global
or local minimizers of the functional $J$ over $\mathcal N$.
However, the phase set $\mathcal X$
is different from $\mathcal N$, and their intersection is just $\mathcal S$
(see~\cite[Proposition 10]{AK09}). Hence it is unclear whether or not
every local minimizer of $J$ over $\mathcal N$ also locally minimizes
$J$ over $\mathcal X$. The following proposition gives an affirmative
answer to this question for \emph{isolated} local minimizers over
$\mathcal N$.

\begin{proposition}\label{P:loc_min}
 Let $\phi$ be an \emph{isolated} local minimizer of $J$ over $\mathcal N$,
 that is, there exists $r_0 > 0$ such that 
\begin{equation}\label{lm:hypo}
J(\phi) < J(w) \quad \mbox{ for all } \ w \in \left( \mathcal N \cap
 B_{H^1_0(\Omega)}(\phi;r_0) \right)\setminus \{\phi\}.
\end{equation} 
Then $\phi$ is also a local minimizer of $J$ over $\mathcal X$.
\end{proposition}

Before proving the proposition above, we note that:

\begin{lemma}\label{L:loc_min}
For each $\phi \in \mathcal S$ the following conditions are
 equivalent\/{\rm :}
\begin{enumerate}
 \item 
       There exists $r_0 > 0$ such that
       $J(\phi) \leq J(w)$ for all $w \in \mathcal N \cap
       B_{H^1_0(\Omega)}(\phi;r_0)$.
 \item 
       There exists $r_0 > 0$ such that
       $R(\phi) \leq R(w)$ for all $w \in \mathcal N \cap
       B_{H^1_0(\Omega)}(\phi;r_0)$.
 \item 
       There exists $r_1 > 0$ such that
       $R(\phi) \leq R(w)$ for all $w \in \mathcal X \cap
       B_{H^1_0(\Omega)}(\phi;r_1)$.
\end{enumerate}
Here we note that one can take the same $r_0$ for {\rm (i)} and {\rm (ii)}.
\end{lemma}

\begin{proof}
 We first note that 
$$
J(w) = \dfrac{m-2}{2m} \lambda_m^{-2/(m-2)}
 R(w)^{2m/(m-2)} \quad \mbox{ for all } \ w \in \mathcal N.
$$
Hence it is obvious that
 (i) and (ii) are equivalent with the same choice of $r_0 > 0$.
 So it remains to prove the equivalence between (ii) and (iii). 
 For each $w \in H^1_0(\Omega) \setminus \{0\}$, set positive constants
$$
x(w) := t_*(w)^{-1/(m-2)}, \quad 
n(w) := \left( \dfrac{\|w\|_{H^1_0(\Omega)}^2}{\lambda_m
 \|w\|_{L^m(\Omega)}^m} \right)^{1/(m-2)}.
$$
Then it follows that $x(w) \leq n(w)$, $x(w)w \in \mathcal X$ and $n(w)w
 \in \mathcal N$ (see~\cite[Proposition 10]{AK09}). 
 First, assume (ii). Let $w \in \mathcal X$ be such
 that $\|w - \phi\|_{H^1_0(\Omega)} < r_1$ with $r_1 > 0$
 which will be determined later. We observe that
\begin{align}
\|n(w)w - \phi\|_{H^1_0(\Omega)}
 &\leq \left| n(w) - 1 \right| \|w\|_{H^1_0(\Omega)} + \|w -
 \phi\|_{H^1_0(\Omega)} \nonumber\\
 &< \left| n(w) - 1 \right|  \left(
 \|\phi\|_{H^1_0(\Omega)} + r_1 \right) + r_1.\label{p1}
\end{align}
Since $n(\cdot)$ is continuous in $H^1_0(\Omega) \setminus \{0\}$ and
 $n(\phi) = 1$, one can take $r_1 > 0$ small enough that the right-hand
 side of \eqref{p1} is less than $r_0$. It follows that
$$
R(\phi) \stackrel{\text{(ii)}}\leq R(n(w)w) = R(w).
$$
Thus (iii) follows. Next assume (iii) and let $w \in \mathcal N$ be such
 that $\|w-\phi\|_{H^1_0(\Omega)} < r_0$ with $r_0 > 0$ to be
 determined. Then one can similarly derive
$$
\|x(w)w - \phi\|_{H^1_0(\Omega)}
 < \left| x(w) - 1 \right|  \left(
 \|\phi\|_{H^1_0(\Omega)} + r_0 \right) + r_0.
$$
So choosing $r_0 > 0$ small enough and employing the continuity of
 $t_*(\cdot)$ in $H^1_0(\Omega)$ along with $t_*(\phi) = 1$, one deduces
 that $\|x(w)w - \phi\|_{H^1_0(\Omega)} < r_1$. Consequently, (iii)
 implies $R(\phi) \leq R(x(w)w) = R(w)$, whence (ii) follows.
\end{proof}

The fact above also holds true for global minimizers (i.e., $r_0 = r_1 =
\infty$). Moreover, it is known (see Proposition \ref{P:X}) that the set of
 (global) minimizers of $J$ over $\mathcal X$ coincides with the set of
 least energy solutions, which can be also formulated as (global)
 minimizers of $J$ over $\mathcal N$ (see, e.g.,~\cite[Chap.~4]{Willem}).

\begin{proof}[Proof of Proposition \ref{P:loc_min}]
 Since $\phi$ is the (unique) minimizer of $J$ over $\mathcal N \cap
 B_{H^1_0(\Omega)}(\phi;r_0)$ by assumption, due to Lemma
 \ref{L:loc_min}, it holds that
\begin{equation}\label{p2}
R(\phi) \leq R(w) \quad \mbox{ for all } \ w \in \mathcal X \cap
 B_{H^1_0(\Omega)}(\phi;r_1) 
\end{equation}
for some $r_1 > 0$. 
 Suppose on the contrary that $\phi$ is not a local minimizer of $J$
 over $\mathcal X$; then for each $n \in \mathbb N$ we can take
 $v_{0,n} \in \mathcal X \cap B_{H^1_0(\Omega)}(\phi;1/n)$ such that
$$
J(v_{0,n}) < J(\phi).
$$
Let $v_n = v_n(x,s)$ be the solution of \eqref{eq:1.6}--\eqref{eq:1.8}
 with the initial data $v_n(0) = v_{0,n}$. Then one observes that
$$
R(v_n(s)) \leq R(v_{0,n}), \quad J(v_n(s)) \leq J(v_{0,n}) < J(\phi)
\quad \mbox{ for all } \ s \geq 0.
$$
Since $v_{0,n}$ belongs to $\mathcal X$, there is $\psi_n\in \mathcal S
 \subset \mathcal N$
 such that $v_n(s) \to \psi_n$ along a subsequence of $s \to \infty$. Thus
 we see that $J(\psi_n) < J(\phi)$, which implies $\psi_n \not\in
 B_{H^1_0(\Omega)}(\phi;r_0)$ by assumption.
 Moreover, it follows that $R(\psi_n) \leq R(v_{0,n})$. 

 Now, let us take $s_n > 0$ such that 
 $$
 \|v_n(s)-\phi\|_{H^1_0(\Omega)} < \vep \quad \mbox{ for all } \
 s \in [0,s_n), \quad \mbox{ and } \quad
 \|v_n(s_n)-\phi\|_{H^1_0(\Omega)} = \vep
$$
with $\vep \in (0, r_0 \wedge r_1)$ which will be determined later \III (cf.~see \eqref{ass_cont}). \EEE
Then since $\mathcal X$ is sequentially closed in the weak topology of
 $H^1_0(\Omega)$ (see Proposition \ref{P:X}), there exists $z \in
 \mathcal X \cap B_{H^1_0(\Omega)}(\phi;r_1)$ such that
$$
 v_n(s_n) \to z \quad \mbox{ weakly in } H^1_0(\Omega)
 \ \mbox{ and \ strongly in } L^m(\Omega).
$$
Therefore by Lemma \ref{L:infR}, we deduce that
$$
\liminf_{n \to \infty} R(v_n(s_n))
=\liminf_{n \to \infty} 
\dfrac{\|\nabla v_n(s_n)\|_{L^2(\Omega)}}{\|v_n(s_n)\|_{L^m(\Omega)}}
\geq \dfrac{\|\nabla z\|_{L^2(\Omega)}}{\|z\|_{L^m(\Omega)}}
= R(z) \stackrel{\eqref{p2}}\geq R(\phi).
$$
On the other hand, recalling $R(v_n(s_n)) \leq R(v_{0,n})$, one has
$$
\limsup_{n \to \infty} R(v_n(s_n)) \leq \lim_{n \to \infty} R(v_{0,n}) =
 R(\phi).
$$
Combining these facts, we obtain
$$
R(v_n(s_n)) \to R(\phi) \quad \mbox{ and } \quad R(z) = R(\phi).
$$

Therefore we see that
\begin{align*}
 \|\nabla v_n(s_n)\|_{L^2(\Omega)}
 &= R(v_n(s_n)) \|v_n(s_n)\|_{L^m(\Omega)}\\
 &\to R(\phi) \|z\|_{L^m(\Omega)} 
 = R(z) \|z\|_{L^m(\Omega)} = \|\nabla z\|_{L^2(\Omega)},
\end{align*}
which along with the uniform convexity of $H^1_0(\Omega)$ implies
$$
v_n(s_n) \to z \quad \mbox{ strongly in } H^1_0(\Omega).
$$
Thus we get
$$
\|z - \phi\|_{H^1_0(\Omega)} = \vep.
$$

Now, $n(z)z$ belongs to $\mathcal N$. We here claim that
\begin{equation}\label{lm:claim}
 0 \neq \|n(z)z - \phi\|_{H^1_0(\Omega)} < r_0
\end{equation}
for sufficiently small $\vep > 0$.
Indeed, repeating the same argument as in the proof of Lemma
 \ref{L:loc_min}, we find that $ \|n(z)z - \phi\|_{H^1_0(\Omega)} < r_0$
for $\vep > 0$ small enough. On the other hand, recall that $z \neq
 \phi$ and $z, \phi \in \mathcal X$. The ray from the origin through $w
 \in H^1_0(\Omega) \setminus \{0\}$, i.e., $\{ k w \in H^1_0(\Omega)
 \colon k > 0 \}$, intersects $\mathcal X$ (resp., $\mathcal N$) only at
 the single point $x(w)w$ (resp., $n(w)w$) (see~\cite[Proposition
 10]{AK09}); therefore $z$ and $\phi$ do not lie on the same ray from
 the origin. Hence one observes that
$$
n(z)z \neq n(\phi)\phi = \phi.
$$
Thus we obtain \eqref{lm:claim}. Recall $R(n(z)z) = R(z) = R(\phi)$ and
 note that
$$
J(n(z)z) 
= \frac{m-2}{2m} \lambda_m^{-2/(m-2)} R(n(z)z)^{2m/(m-2)}
= \frac{m-2}{2m} \lambda_m^{-2/(m-2)} R(\phi)^{2m/(m-2)}
= J(\phi).
$$
However, these facts yield a contradiction to the
 assumption \eqref{lm:hypo}. The proof is completed.
\end{proof}

The inverse relation can be easily proved without imposing any additional
assumption.

\begin{proposition}
 Let $\phi$ satisfy \eqref{locmin}. Then $\phi$ locally minimizes $J$
 over $\mathcal N$.
\end{proposition}

\begin{proof}
Assume that $\phi$ satisfies \eqref{locmin}. 
As in the proof of Lemma \ref{L:loc_min}, one can choose
 $\delta > 0$ small enough that $x(w)w \in \mathcal X \cap
 B_{H^1_0(\Omega)}(\phi ; r_0)$ for all $w \in \mathcal N \cap
 B_{H^1_0(\Omega)}(\phi;\delta)$. Then by assumption, we obtain
$$
J(\phi) \leq J(x(w)w) \quad \mbox{ for all } \ 
w \in \mathcal N \cap B_{H^1_0(\Omega)}(\phi;\delta).
$$
On the other hand, by the definition of $\mathcal N$, it holds that
 $J(w) = \sup_{c > 0} J(cw)$ for each $w \in \mathcal
 N$. Hence it follows that
$$
J(\phi) \leq J(x(w)w) \leq J(w) \quad \mbox{ for all } \ w \in \mathcal N \cap
 B_{H^1_0(\Omega)}(\phi ; \delta).
$$
Thus we conclude that $\phi$ is a local minimizer of $J$ over $\mathcal N$.
\end{proof}

\section{Instability of positive radial profiles in thin annular domains}

In the final section, we shall apply the \L ojasiewicz-Simon inequality to prove the instability of sign-definite asymptotic profiles which do not attain local minima of $J$ over $\mathcal X$. Then one can prove Theorem \ref{T:inst}, that is, the instability of the positive radial asymptotic profile (equivalently, the positive radial solution of \eqref{eq:1.10}, \eqref{eq:1.11}) in the annular domain
$$
\Omega = \left\{
x \in \mathbb R^N \colon a < |x| < b
\right\}
$$ 
with $0 < a < b < \infty$ satisfying \eqref{hypo:inst}, as a corollary.

\begin{theorem}[Instability of sign-definite profiles except for local minimizers of $J$ over $\mathcal X$]\label{T:inst-nlm}
 Let $\phi$ be a positive {\rm (}or negative{\rm )} solution of \eqref{eq:1.10}, \eqref{eq:1.11} which does not attain any local minimum of $J$ over $\mathcal X$. Then $\phi$ is an unstable asymptotic profile for FDE {\rm (}in the sense of Definition \ref{D:stbl}{\rm )}.
\end{theorem}

\begin{proof}
 Let $\phi$ be a positive (or negative) solution of \eqref{eq:1.10}, \eqref{eq:1.11} such that $\phi$ does not attain any local minimum of $J$ over $\mathcal X$, that is, there exists a sequence $(v_{0,n})$ in $\mathcal X$ such that $J(v_{0,n}) < J(\phi)$ and $v_{0,n} \to \phi$ strongly in $H^1_0(\Omega)$. Then by strong maximum principle and elliptic regularity, $\phi$ also satisfies \eqref{smp}; therefore the \L ojasiewicz-Simon inequality, i.e., Lemma \ref{L:LS}, is valid for $\phi$ as well. Since $v_{0,n}$ lies on $\mathcal X$, one can take a nontrivial solution $\psi_n$ of \eqref{eq:1.10}, \eqref{eq:1.11} such that the solution $v_n$ of \eqref{eq:1.6}--\eqref{eq:1.8} with $v_0 = v_{0,n}$ converges to $\psi_n$ strongly in $H^1_0(\Omega)$ along a subsequence of $s \to \infty$. From the non-increase of the energy, one has $J(\psi_n) < J(\phi)$. Now, suppose that $\psi_n$ converges to $\phi$ strongly in $H^1_0(\Omega)$ as $n \to \infty$. Then by utilizing elliptic regularity technique, one can check that $\psi_n$ converges to $\phi$ in $C^2(\overline\Omega)$; in particular, $\|\psi_n\|_{L^\infty(\Omega)} \leq \|\phi\|_{L^\infty(\Omega)} + 1 =: L$ for $n > 0$ large enough. Hence thanks to Lemma \ref{L:LS}, since $J'(\psi_n) = 0$, for sufficiently large $n$, $\psi_n$ must take the same critical value as $\phi$, that is, $J(\psi_n) = J(\phi)$. However, it is a contradiction to the difference of the energy. Therefore $(\psi_n)$ does not converge to $\phi$ in $H^1_0(\Omega)$ as $n \to \infty$.

Hence one can take $\delta_1 > 0$ and a subsequence $(n_k)$ of $(n)$ such that $\|\psi_{n_k} - \phi\|_{H^1_0(\Omega)} \geq \delta_1$ for all $k \in \mathbb N$. Therefore for each $k \in \mathbb N$, the solution $v_{n_k}(s)$ of \eqref{eq:1.6}--\eqref{eq:1.8} for the initial data $v_{0,n_k}$ must go away from the neighborhood $B_{H^1_0(\Omega)}(\phi;\delta_1/2)$ for $s > 0$ sufficiently large. On the other hand, $v_{0,n_k} \to \phi$ strongly in $H^1_0(\Omega)$ as $k \to \infty$. Thus we have proved the instability of $\phi$.
\end{proof}

Theorem \ref{T:inst} follows from the theorem stated above.

\begin{proof}[Proof of Theorem \ref{T:inst}]
Let $\phi$ be the positive radial solution of \eqref{eq:1.10}, \eqref{eq:1.11}. Then as in~\cite{AK12}, for $\vep > 0$ small enough, one can explicitly construct $v_{0,\vep} \in \mathcal X$ such that
$$
J(v_{0,\vep}) < J(\phi) \quad \mbox{ and } \quad
v_{0,\vep} \to \phi \ \mbox{ strongly in } H^1_0(\Omega)
\ \mbox{ as } \ \vep \to 0
$$
under the assumption \eqref{hypo:inst}. This fact yields that $\phi$ is not a local minimizer of $J$ over $\mathcal X$; thus the instability of $\phi$ follows from Theorem \ref{T:inst-nlm}.
\end{proof}

\section*{Acknowledgments}

The author is supported by the Carl Friedrich von Siemens Stiftung and the Alexander von Humboldt Stiftung through a Research Fellowship for Experienced Researchers. Moreover, he is also supported by JSPS KAKENHI Grant Number 25400163 and by the JSPS-CNR bilateral joint research project: \emph{Innovative Variational Methods for Evolution Equations}.

 \III 
\appendix

\section{Derivation of \eqref{apdx1} and \eqref{apdx2}}

This appendix is devoted to verifying \eqref{apdx1} and \eqref{apdx2}. Formally test \eqref{eq:1.1} by $\partial_t u$ to obtain
$$
\dfrac 1 2 \dfrac{\d}{\d t} \|u(t)\|_{H^1_0(\Omega)}^2 \leq 0 \quad \mbox{ for a.e. } \ t > 0
$$
(it can be justified as in~\cite{G:EnSol}). Then
$$
\sup_{t \geq 0} \|u(t)\|_{H^1_0(\Omega)} \leq \|u_0\|_{H^1_0(\Omega)}.
$$
By using Tartar's inequality along with the fact that $|u|^{m-2}u \in C([0,T];H^{-1}(\Omega))$, as in \eqref{tartar}, we see that
$$
u \in C([0,T];L^m(\Omega)).
$$
Moreover, recalling Lemma 8.1 of~\cite{LM}, we deduce that
$$
u \in C_w([0,T];H^1_0(\Omega)).
$$
Hence it follows that, for each $s \in [0,T]$,
$$
\|u(s)\|_{H^1_0(\Omega)} \leq \liminf_{t \to s} \|u(t)\|_{H^1_0(\Omega)}.
$$
Since $t \mapsto \|u(t)\|_{H^1_0(\Omega)}$ is non-increasing, we have
$$
\lim_{t \searrow s} \|u(t)\|_{H^1_0(\Omega)} = \|u(s)\|_{H^1_0(\Omega)}.
$$
Thus by the uniform convexity of $H^1_0(\Omega)$,
$$
u(t) \to u(s) \quad \mbox{ strongly in } H^1_0(\Omega) \ \mbox{ as } \ t \searrow s,
$$
which implies $u \in C_+([0,T];H^1_0(\Omega))$. Finally, by comparison of both sides of \eqref{eq:1.1}, since $- \Delta : H^1_0(\Omega) \to H^{-1}(\Omega)$ is an isomorphism, $\partial_t (|u|^{m-2}u)$ belongs to $C_+([0,T];H^{-1}(\Omega))$.

 \EEE 

\pagestyle{myheadings}

\end{document}